\documentclass[amssymb,twoside,12pt]{article}
\usepackage{latexsym,amsmath,graphicx,mathrsfs,amssymb}
\usepackage{amsfonts}
\usepackage{enumitem}
\newtheorem{theorem}{Theorem}[section]
\newtheorem{lemma}[theorem]{{\bf Lemma}}

\newtheorem{rem}[theorem]{{\bf Remark}}

\newtheorem{definition}{Definition}[section]
\numberwithin{equation}{section}
\newenvironment{proof}{\indent{\em Proof:}}{\quad \hfill
$\Box$\vspace*{2ex}}

\begin{document}
\setcounter{page}{1}
\begin{center}
\vspace{0.3cm} {\large{\bf Ulam-Hyers stabilities of mild solutions of the fractional nonlinear abstract Cauchy problem}} \\
\vspace{0.4cm}
 J. Vanterler da C. Sousa $^1$ \\
vanterler@ime.unicamp.br  \\

\vspace{0.30cm}
K. D. Kucche $^2$\\
kdkucche@gmail.com\\

\vspace{0.30cm}
E. Capelas de Oliveira $^3$\\
capelas@ime.unicamp.br\\
\vspace{0.35cm}

\vspace{0.30cm}
$^{1,3}$ Department of Applied Mathematics, Imecc-Unicamp,\\ 13083-859, Campinas, SP, Brazil.\\

\vspace{0.30cm}
{$^{2}$ Department of Mathematics, Shivaji University, Kolhapur,\\ Maharashtra 416 004, India}
\end{center}

\def\baselinestretch{1.0}\small\normalsize
\begin{abstract}
Since the main work on Ulam-Hyers dependable stabilities of differential equations to date, numerous significant and applicable papers have been published, both in the sense of integer order and fractional order differential equations. However, when we enter the field of fractional differential equations, the path that is still long to be traveled, although there is a range of published works. In this sense, in this paper, we will investigate the Ulam--Hyers and Ulam--Hyers--Rassias stabilities of mild solutions of the fractional nonlinear abstract Cauchy problem on the intervals $[0,T]$ and $[0,\infty)$, by means of Banach fixed point theorem.
\end{abstract}
\allowdisplaybreaks

\section{Introduction}

One of the most dynamic subjects of differential equations has been the stability theory of Ulam-Hyers. The theme came in 1940 by Ulam in a lecture on unresolved issues at the University of Wisconsin \cite{16,17}. The issue raised by Ulam was partially answered the following year by Hyers in the case of the Banach spaces. Thus, the theory of stabilities, came to be called Ulam-Hyers. However, in 1978 \cite{21}, Rassias introduced a generalization of the version exhibited by Hyers. In this sense, due to this breakthrough and novelty in mathematical analysis, numerous specialists have researched  the stability of solutions of functional differential equations. The idea of Ulam-Hyers stability for functional equations, is the substitution of the functional equation for a given inequality that acts as a perturbation of the equation. We recommend a few monographs and papers that permit a progressively careful study of the subjects
 \cite{akkouchi,18,19,20}.

With the beginning of the fractional calculus and over the years his theory being well consolidated and grounded, many researchers began to look in a different way for the area, especially researchers working with differential equations \cite{ZE1,ZE23,ZE3,zhou,yang,kilbas,samko}. In this sense, today it is more than proven that investigating and analyzing certain physical problems, through fractional derivatives, ensures more accurate and consistent results with reality. On the other hand, moving to a more theoretical side, investigating the existence, uniqueness and Ulam-Hyers stability of solution of fractional differential equations has gained increasing prominence in the scientific community, although there are a range of works, the theory is still being built with good results \cite{wang222,wangnew,wangclass,wanglinear}.

In 2012, Wang and Zhou \cite{wang222} in their work, investigated several kind of stabilities of the mild solution stability of the fractional evolution equation in Banach space, namely: Mittag-Leffler-Ulam stability, Mittag-Leffler-Ulam- Leffler-Ulam-Hyers stability, Mittag-Leffler-Ulam-Hyers - Rassias stability and generalized Mittag-Leffler-Ulam-Hyers-Rassias stability. In 2014, Abbas \cite{abbas1} investigated the existence, uniqueness, and stability of the mild solution of the integrodifferential equation with nonlocal conditions through Holder's inequality, Schauder's fixed point theorem, and Gronwall's inequality in Banach space. Other work can be found in the references of the two papers themselves. On the other hand, Zhou and Jiao \cite{zhou}, using fractional operators and some fixed point theorems, investigated the existence and uniqueness of mild solutions of fractional neutral evolution equations and made some applications in order to elucidate the obtained results. In this sense, Saadati et. al. \cite{saadati}, presented results on the existence of mild solutions for fractional abstract equations with non-instantaneous impulses. In order to obtain such results, the authors used non-compactness measure and the Darbo-Sadovskii and Tichonov fixed point theorems . For a more in-depth reading, we suggest some papers \cite{yang,sousa2,dabas1,jawahdou,balac,olszowy,chen}.

Although we are faced with a significant amount of work dealing with solution properties of fractional differential equations, there is still much work to be done. In order to propose new results and provide new materials on Ulam-Hyers stability and to contribute positively to the area, the present paper has as main objective to investigate the Ulam-Hyers stabilities on the intervals $[0,T]$ and $[0,\infty)$.

So let's consider the fractional nonlinear abstract Cauchy problem given by
\begin{equation}\label{CP}
\left\{
\begin{array}{rll}
\displaystyle {}^{H}{\mathbb{D}}_{0^{+}}^{\alpha,\beta} \xi(t) & = & \mathcal{A} \xi(t) + u(t) \mathcal{H}(t,\xi(t)), ~t \in I\\
I_{0^{+}}^{1-\gamma} \xi(0) & = & \xi_0 
\end{array} \right.
\end{equation}
where ${}^{H}{\mathbb{D}}_{0^{+}}^{\alpha,\beta} (\cdot)$ is the Hilfer fractional derivative of order $0 < \alpha \leq 1$ and type $0 \leq \beta \leq 1$, $\gamma=\alpha+\beta-\alpha \beta$,  $I=[0,T]$ or $[0,\infty)$, $\xi \in C(I,\Omega)$, $\Omega:=(\Omega,\|\cdot\|)$ is a Banach space, $t \in I$, $ \mathcal{A}:\Omega \rightarrow \Omega$ is the infinitesimal generator of a $C_0$-semigroup $(\mathbb{S}(t))_{t \geq 0}$ and $\mathcal{H}: I \times \Omega \rightarrow \Omega$ is a given continuous function.

We highlight below the main points that motivated us to investigate the mild solution stability for the fractional abstract Cauchy problem:
\begin{enumerate}
    \item A new class of Ulam-Hyers type stabilities for the fractional abstract Cauchy problem;
    
    \item At the limit of $\beta\rightarrow 1$ in the mild solution of the abstract Cauchy problem with $0 <\alpha <1$, we have a sub-class of Ulam-Hyers stabilities for the Riemann-Liouville fractional derivative;
    
    \item At the limit of $\beta\rightarrow 0$ in the mild solution of the abstract Cauchy problem with $0 <\alpha <1$, we have a sub-class of Ulam-Hyers stabilities for the fractional derivative of Caputo;
    
    \item When $\alpha=1$, we have as particular case, the integer version;
    
    \item An important consequence of the obtained results are the possible future applications through the Ulam-Hyers stabilities in engineering, biology and especially in mathematics;
    \end{enumerate}

The paper is organized as follows. In section 2, we introduce the $\psi$-Riemann-Liouville fractional integral, the $\psi$-Hilfer fractional derivative and fundamental concept of the operator $(\alpha,\beta)$- resolvent. In this sense, it is presented the mild solution of the fractional Cauchy problem as well as the Ulam-Hyers stability. In section 3, it is directed to the first result of this paper, that is, we investigate the Ulam-Hyers and Ulam-Hyers-Rassias stabilities in the $[0, T]$ range and discuss some particular cases. In section 4, we discuss the Ulam-Hyers and Ulam-Hyers-Rassias stabilities in the interval $[0,\infty)$. Concluding remarks close the paper.


\section{Preliminaries}

In this section, we will introduce some important definitions and results in order to assist in the development of this paper.

Let $T > 0$ be a given positive real number. The weighted space of continuous functions $\xi \in I'=(0,T]$ is given by \cite{sousa21}
{
\begin{equation*}
C_{1-\gamma}(I,\Omega)= \left\{ \xi \in C(I',\Omega), \, t^{1-\gamma} \xi(t) \in C(I,\Omega) \right\} 
\end{equation*}
where $0 < \gamma \leq 1$, with norm

\begin{equation*}
\begin{array}{rll}
||\xi||_{C_{1-\gamma}} & = & \displaystyle \sup_{t \in I} ||t^{1-\gamma}\xi(t)||
\end{array}
\end{equation*}
and
\begin{equation*}
\begin{array}{rll}
||\xi- \phi||_{C_{1-\gamma}} & = &  {\rm{d}}_{1-\gamma} (\xi,\phi) :=\displaystyle \sup_{t \in I} ||t^{1-\gamma}(\xi(t)-\phi(t))|| \cdot
\end{array}
\end{equation*}

Let $\left( a,b\right) $ $\left( -\infty \leq a<b\leq \infty \right) $ be a finite interval {\rm{(or infinite)}} of the real line $\mathbb{R}$ and let $\alpha >0$. Also let $\psi \left( x\right) $ be an increasing and positive monotone function on $\left( a,b\right] ,$ having a continuous derivative $\psi ^{\prime }\left( x\right)$ {\rm{(we denote first derivative as $\dfrac{d}{dx}\psi(x)=\psi'(x)$)}} on $\left( a,b\right) $. The left-sided fractional integral of a function $f$ with respect to a function $\psi $ on $ \left[ a,b\right] $ is defined by \cite{ZE1,sousa21}
\begin{equation}\label{eq7}
\mathcal{I}_{a+}^{\alpha ;\psi }f\left( x\right) =\frac{1}{\Gamma \left( \alpha
\right) }\int_{a}^{x}\psi ^{\prime }\left( s\right) \left( \psi \left(
x\right) -\psi \left( s\right) \right) ^{\alpha -1}f\left( s\right) ds.
\end{equation}

On the other hand, let $n-1<\alpha <n$ with $n\in \mathbb{N},$ let $J=\left[ a,b\right] $ be an interval such that $-\infty \leq a<b\leq \infty $ and let $f,\psi \in C^{n}\left[ a,b\right] $ be two functions such that $\psi $ is increasing and $\psi ^{\prime }\left( x\right) \neq 0,$ for all $x\in J$. The left-sided $\psi -$Hilfer fractional derivative $^{H}\mathbb{D}_{a+}^{\alpha ,\beta ;\psi }\left( \cdot \right) $ of a function $f$ of order $\alpha $ and type $0\leq \beta \leq 1,$ is defined by \cite{ZE1,ZE23}
\begin{equation}\label{eq8}
^{H}\mathbb{D}_{a+}^{\alpha ,\beta ;\psi }f\left( x\right) =\mathcal{I}_{a+}^{\beta \left( n-\alpha \right) ;\psi }\left( \frac{1}{\psi ^{\prime }\left( x\right) }\frac{d}{dx}\right) ^{n}\mathcal{I}_{a+}^{\left( 1-\beta \right) \left(
n-\alpha \right) ;\psi }f\left( x\right) .
\end{equation}
Let $(\Omega,||\cdot||)$ be a given Banach space and $I=[0,+\infty)$ or $I=[0,T]$ where $T$ and $\mathscr{L}(\Omega)$ the set of bounded linear maps from $\Omega$ to $\Omega$.

Next, we present the definition of the fundamental operator $(\alpha,\beta)$-resolvent in the presentation of the mild solution of the fractional abstract Cauchy problem Eq.(\ref{CP}).

\begin{definition} {\rm \cite{chen}} Let $\alpha > 0$ and $\beta \geq 0$. A function $\mathbb{S}_{\alpha,\beta} : \mathbb{R}_{+} \to \mathscr{L}(\Omega)$ is called a $\beta$-times integrated $\alpha$-resolvent operator function  or an $(\alpha,\beta)$-resolvent operator function {\rm{(ROF)}} if the following conditions are satisfied:

\begin{tabular}{cl}
{\rm{(A)}} & $\mathbb{S}_{\alpha,\beta}(\cdot)$ is strongly continuous on $\mathbb{R}_{+}$ and $\mathbb{S}_{\alpha,\beta}(0)=g_{\beta+1}(0)I$;\\
{\rm{(B)}} & $\mathbb{S}_{\alpha,\beta}(s) \mathbb{S}_{\alpha,\beta}(t)= \mathbb{S}_{\alpha,\beta}(t) \mathbb{S}_{\alpha,\beta}(s)$ for all $t,s \geq 0$;\\
{\rm{(C)}} & the function equation 
\\
& 
$
\mathbb{S}_{\alpha,\beta}(s) I_t^{\alpha} \mathbb{S}_{\alpha,\beta}(t) - I_s^{\alpha} \mathbb{S}_{\alpha,\beta}(s)\mathbb{S}_{\alpha,\beta}(t) $ $=g_{\beta+1}(s) I_t^{\alpha} \mathbb{S}_{\alpha,\beta}(t) - g_{\beta+1}(t) I_s^{\alpha} \mathbb{S}_{\alpha,\beta}(s)$ 
\\
& holds for all $t,s \geq 0$.
\end{tabular}

\medskip
\end{definition}

The generator $\mathcal{A}$ of $\mathbb{S}_{\alpha,\beta}$ is defined by
\begin{equation}
D(\mathcal{A}):= \left\{x \in \Omega: \lim_{t \to 0^{+}} \frac{\mathbb{S}_{\alpha,\beta}(t)\, x - g_{\beta+1}(t)\, x}{g_{\alpha+\beta+1}(t)} \,\, {\rm{exists}}    \right\}
\end{equation}
and 
\begin{equation}
\mathcal{A}\,x := \lim_{t \to 0^{+}} \frac{\mathbb{S}_{\alpha,\beta}(t)\, x - g_{\beta+1}(t)\, x}{g_{\alpha+\beta+1}(t)}\, , \quad x \in D(\mathcal{A}),
\end{equation}
where $g_{\alpha+\beta+1}(t):= \dfrac{t^{\alpha+\beta}}{\Gamma(\alpha+\beta)}$ ($\alpha+\beta>0$).

An $(\alpha,\beta)$-ROF $\mathbb{S}_{\alpha,\beta}$ is said to be exponentially bounded if there exist constants $\delta \geq 1$, $w \geq 0$ such that
$||\mathbb{T}_{\alpha}(t)|| \leq  \delta\, e^{wt}$
and $||\mathbb{S}_{\alpha,\beta}(t)||  \leq  \delta\, e^{wt}, ~t \geq 0$.

Now, we consider the continuous function given $\mathcal{H}: I \times \Omega \rightarrow \Omega$ such that, for almost all $t \in I$, we get
\begin{equation}\label{eq2}
||\mathcal{H}(t,x) - \mathcal{H}(t,y)|| \leq \ell (t) ||x-y||_{C_{1-\gamma}} \, , ~ x,y \in \Omega
\end{equation}
where $\ell:[0,T] \to \mathbb{R}^{+}$ and $u:[0,T] \to \mathbb{R}$ are two given measurable functions such that $\ell,u$ and $\ell u$ are locally integrable on $I$.

The following is the definition of the Mainardi function, fundamental in mild solution of the Eq.(\ref{CP}). Then, the Wright function, denoted by $M_{\alpha}(\theta)$, is defined by \cite{sousa20,gu}
\begin{equation*}
M_{\alpha}(\theta) = \sum_{n=1}^{\infty} \frac{(-\theta)^{n-1}}{(n-1)!\Gamma(1-\alpha n)}\, , \quad 0 < \alpha < 1, \quad \theta \in \mathbb{C}
\end{equation*}
satisfying the relation
\begin{equation*}
\int_0^{\infty} \theta^{\overline{\delta}} M_{\alpha} (\theta) \, {\rm{d}}\theta = \frac{\Gamma(1+\overline{\delta})}{\Gamma(1+\alpha \overline{\delta})}\, , \quad {\rm{for}} \,\,
\theta, \overline{\delta} \geq 0 \cdot
\end{equation*}

\begin{lemma} {\rm \cite{sousa20,gu}} The fractional nonlinear differential equation, {\rm{Eq.(\ref{CP})}}, is equivalent to the integral equation
\begin{equation}\label{EI}
\xi(t) = \frac{t^{\gamma-1}}{\Gamma(\gamma)}\xi(0) + \frac{1}{\Gamma(\alpha)} \int_0^t (t-s)^{\alpha -1} \left[\mathcal{A} \xi(s) + u(s) \, \mathcal{H}(s,\xi(s))  \right]\, {\rm{d}}s\, , \,\, t \in [0,T] \cdot
\end{equation}

A function $\xi \in C_{1-\gamma}(I,\Omega)$ is called a mild solution of {\rm{Eq.(\ref{CP})}}, if the integral equation, {\rm{Eq.(\ref{EI})}} holds, we have 
\begin{equation}\label{eq4}
\xi(t) = \mathbb{S}_{\alpha,\beta}(t) \xi(0) + \int_0^t \mathbb{T}_{\alpha}(t-s) u(s) \mathcal{H}(s,\xi(s))\, {\rm{d}}s\, , \quad t \in I
\end{equation}
where $\displaystyle \mathbb{T}_{\alpha}(t) = t^{\alpha -1} G_{\alpha}(t)$, $\displaystyle G_{\alpha}(t) = \int_0^{\infty} \alpha \theta M_{\alpha}(\theta) \mathbb{S}(t^{\alpha \theta})\, {\rm{d}}\theta$ and $\mathbb{S}_{\alpha,\beta}(t) = \mathcal{I}_0^{\beta(1 -\alpha)} \mathbb{T}_{\beta}(t)$.
\end{lemma}

For a given $\xi_0 \in \Omega$ and any $\xi \in C_{1-\gamma}(I,\Omega)$, we set
\begin{equation}\label{eq5}
\Lambda (\xi)(t):= \mathbb{S}_{\alpha,\beta}(t) \xi_0 + \int_0^t \mathbb{T}_{\alpha}(t-s) u(s) \mathcal{H}(s,\xi(s))\, {\rm{d}}s
\end{equation}
for all $t \in I$.

For the procedure in this paper, $\ell,u$ are measurable functions such that $\ell,u$ and the product $\ell u$ are locally integrable. Moreover, it is easy to see that the application $\xi \to \Lambda(\xi)$ is a self-mapping of the space $C_{1-\gamma}(I,\Omega)$.

For $\xi_{0}\in\Omega$ and $\varepsilon$, we consider
\begin{equation}\label{2.1}
\xi(t) = \Lambda (\xi(t)), \quad t \in I
\end{equation}
and the following inequalities
\begin{equation}\label{2.2}
||t^{1-\gamma}\left( \xi(t) - \Lambda(\xi(t))\right) || \leq \varepsilon \, , \quad t \in I
\end{equation}
and 
\begin{equation}\label{2.3}
||t^{1-\gamma}\left( \xi(t) - \Lambda(\xi(t))\right) || \leq G(t) \, , \quad t \in I,
\end{equation}
where $\xi\in C_{1-\gamma}(I,\Omega)$ and $G \in C(I,(0,+\infty))$.

The following are the definitions of the main results to be investigated in this paper. Following the methodology of \cite{akkouchi}, the definitions were adapted to the problem version of fractional differential equations. Then we have:

\begin{definition}{\rm \cite{akkouchi,sousaulam}} The {\rm{Eq.(\ref{2.1})}} is Ulam-Hyers stable if there exists a real number $ c > 0$ such that for each $\varepsilon > 0$ and for each solution $\xi \in C_{1-\gamma}(I,\Omega)$ of {\rm{Eq.(\ref{2.2})}} there exists a solutions 
$ v \in C_{1-\gamma}(I,\Omega)$ of {\rm{Eq.(\ref{2.1})}} such that
\begin{equation}
||t^{1-\gamma}\left(\xi(t) - v(t)\right) || \leq \varepsilon \, , \quad t \in I .
\end{equation}
\end{definition}

\begin{definition}{\rm \cite{akkouchi,sousaulam}} The {\rm{Eq.(\ref{2.1})}} is generalized Ulam-Hyers stable if there exists $\theta \in C_{1-\gamma}([0,+\infty), [0,+\infty))$, $\theta(0)=0$, such that for each $\varepsilon > 0$ and for each solution $ \xi \in C_{1-\gamma}(I,\Omega)$ of {\rm{Eq.(\ref{2.2})}} there exists a solutions 
$ v \in C_{1-\gamma}(I,\Omega)$ of {\rm{Eq.(\ref{2.1})}} such that 

\begin{equation}
||t^{1-\gamma}\left(\xi(t) - v(t)\right) || \leq \theta(\varepsilon) \, , \quad t \in I .
\end{equation}
\end{definition}

\begin{definition}{\rm \cite{akkouchi,sousaulam}} The {\rm{Eq.(\ref{2.1})}} is generalized Ulam-Hyers-Rassias stable with respect to $G \in C_{1-\gamma}([0,+\infty),[0,+\infty))$, if there exists $c_{G} > 0$ such that for each solution $\xi \in C_{1-\gamma}(I,\Omega)$ of {\rm{Eq.(\ref{2.3})}} there exists a solution $ v \in C_{1-\gamma}(I,\Omega)$ of {\rm{Eq.(\ref{2.1})}} such that

\begin{equation}
||t^{1-\gamma}\left(\xi(t) - v(t)\right) ||\leq c_G G(t) \, , \quad t \in I \cdot
\end{equation}
\end{definition}
\section{Ulam-Hyers and Ulam-Hyers-Rassias stabilities of mild on $[0,T]$.}

In this section, we investigate the first of the main results of this paper, i.e., the Ulam-Hyers and Ulam-Hyers-Rassias stabilities of Eq.(\ref{2.1}) in the interval $[0,T]$, using the Banach's fixed point theorem.

Let $\left(\mathbb{S}_{\alpha,\beta}(t) \right)_{t \geq 0}$ the $(\alpha,\beta)$-resolvent operator function on a Banach space  $(\Omega,||\cdot||_{C_{1-\gamma}})$. and the continuous function $\xi:[0,T] \to \Omega$, given by
\begin{equation}\label{eq12}
\Lambda(\xi)(t):= \mathbb{S}_{\alpha,\beta}(t) \xi_0 + \int_0^t \mathbb{T}_{\alpha}(t-s) u(s) \mathcal{H}(s,\xi(s))\, {\rm{d}}s \, , \quad t \in [0,T)
\end{equation}
for $\xi_0 \in \Omega$ fixed.

Then, we have the theorem that gives certain conditions, guarantees the Ulam-Hyers stability to Eq.(\ref{2.1}) on the finite interval $[0,T]$.

\begin{theorem} 
Let $\left(\mathbb{S}_{\alpha,\beta}(t) \right)_{t \geq 0}$ the $(\alpha,\beta)$-resolvent operator function on a Banach space $(\Omega,||\cdot||_{C_{1-\gamma}})$, 
with $0 \leq \gamma \leq 1$ and let $T > 0$ be a positive real number. We set
\begin{equation}
\widetilde{\lambda}:= \delta \, T ^{1-\gamma} \int_0^T e^{w(T-s)} |u(s)| \ell(s) \, {\rm{d}}s \cdot
\end{equation}

If $\widetilde{\lambda} < 1$, then the {\rm{Eq.(\ref{2.1})}} is stable in the Ulam-Hyers sense.
\end{theorem}

\begin{proof} 
Admit that $\widetilde{\lambda} < 1$ and let $\varepsilon > 0$ be given. For $\phi,\xi \in C_{1-\gamma}(I,\Omega)$, we have
\begin{equation*}
\begin{array}{rll}
||(\Lambda \phi)(t) - (\Lambda \xi)(t)|| & = & \displaystyle \left|\left|\mathbb{S}_{\alpha,\beta}(t) \xi_0 + \int_0^t \mathbb{T}_{\alpha}(t-s) u(s) \mathcal{H}(s,\phi(s))\, {\rm{d}}s \right.\right.\\
& - & \displaystyle \left.\left. \mathbb{S}_{\alpha,\beta}(t) \xi_0 - \int_0^t \mathbb{T}_{\alpha}(t-s) u(s) \mathcal{H}(s,\xi(s))\, {\rm{d}}s \right|\right|\\
& = & \displaystyle \left|\left| \int_0^t \mathbb{T}_{\alpha}(t-s) u(s) \left( \mathcal{H}(s,\phi(s)) - \mathcal{H}(s,\xi(s)) \right) \, {\rm{d}}s \right|\right|\\
& \leq & \displaystyle \int_0^t ||\mathbb{T}_{\alpha}(t-s)|| |u(s)| ||\mathcal{H}(s,\phi(s)) - \mathcal{H}(s,\xi(s)) || \, {\rm{d}}s\\
& \leq & \displaystyle \int_0^t ||\mathbb{T}_{\alpha}(t-s)|| |u(s)|\, \ell(s)\, ||\phi - \xi ||_{C_{1-\gamma}} \, {\rm{d}}s\\
& = & \displaystyle \delta \int_0^T  \,e^{w(T-s)} |u(s)|\, \ell(s)\, {\rm{d}}s \, ||\phi - \xi ||_{C_{1-\gamma}} \, , \quad t \in [0,T] \cdot
\end{array}
\end{equation*}
Therefore 
\begin{align*}
||(\Lambda \phi) - (\Lambda \xi)||_{C_{1-\gamma}} &= \sup_{t \in I}||t^{1-\gamma}\left((\Lambda \phi)(t) - (\Lambda \xi)(t)\right) ||\\
& \leq  \left( \displaystyle \delta \, T ^ {1-\gamma}\int_0^T  \,e^{w(T-s)} |u(s)|\, \ell(s)\, {\rm{d}}s\right)  \, ||\phi - \xi ||_{C_{1-\gamma}}.
\end{align*}

So, we get
\begin{equation*}
{\rm{d}}_{1-\gamma}(\Lambda \phi, \Lambda \xi) \leq \widetilde{\lambda} \, {\rm{d}}_{1-\gamma}(\phi,\xi).
\end{equation*}

Since $\widetilde{\lambda}<1$,  $\Lambda$ is a contradiction. On the other hand, consider $\theta,\phi \in C_{1-\gamma}(I,\Omega)$, such that
\begin{equation*}
{\rm{d}}_{1-\gamma}(\Lambda \theta, \theta) \leq \varepsilon \,
\end{equation*}
and
\begin{equation*}
{\rm{d}}_{1-\gamma}(\phi,\xi) \leq \frac{\varepsilon}{1-\widetilde{\lambda}}.
\end{equation*}

Then, we obtain
\begin{equation*}
\begin{array}{rll}
{\rm{d}}_{1-\gamma} (\theta,\Lambda \phi) & \leq & {\rm{d}}_{1-\gamma}(\theta, \Lambda \theta) + {\rm{d}}_{1-\gamma}(\Lambda \theta, \Lambda \phi)\\
& \leq & \displaystyle \varepsilon + \frac{\widetilde{\lambda}\varepsilon}{1-\widetilde{\lambda}} \leq \frac{\varepsilon}{1-\widetilde{\lambda}} \cdot
\end{array}
\end{equation*}

In this sense, we have that the  closed ball $\displaystyle \overline{B}_{C_{1-\gamma}} \left(\theta, \frac{\varepsilon}{1-\widetilde{\lambda}} \right)$ of the Banach space $C_{1-\gamma}(I,\Omega)$ is invariant by the map $\Lambda$, i.e.
\begin{equation*}
\Lambda \left( \overline{B}_{C_{1-\gamma}} \left(\theta, \frac{\varepsilon}{1-\widetilde{\lambda}} \right) \right) \subset \overline{B}_{C_{1-\gamma}} 
\left(\theta, \frac{\varepsilon}{1-\widetilde{\lambda}} \right) \cdot 
\end{equation*}

Then, applying the Banach fixed-point theorem to $\Lambda $ acting on $\overline{B}_{C_{1-\gamma}} \left(\theta, \dfrac{\varepsilon}{1-\widetilde{\lambda}} \right)$, we have that there is only one element $\xi \in \overline{B}_{C_{1-\gamma}} 
\left(\theta, \dfrac{\varepsilon}{1-\widetilde{\lambda}} \right)$ such that $\xi=\Lambda (\xi)$.  So we have $\xi$ is a solution of the  {\rm Eq.(\ref{2.1})}, which satisfies 
\begin{equation*}
d_{1-\gamma}(\theta,\xi) \leq \frac{\epsilon}{1-\widetilde{\lambda}},
\end{equation*}
which gives 
\begin{equation*}
||t^{1-\gamma}\left(\theta(t) - \xi(t)\right) || \leq  c \, \varepsilon \, , \quad t \in [0,T]
\end{equation*}
where $c:=1/(1-\widetilde{\lambda})$. Thus, we conclude that the integral equation {\rm{Eq.(\ref{2.1})}} is stable in the Ulam-Hyers sense .
\end{proof}

 Next, we will investigate the Ulam-Hyers-Rassias stability by completing the first purpose of this paper.

\begin{theorem} Let $(\Omega,||\cdot||)$ be a Banach space and let $(\mathbb{S}_{\alpha,\beta}(t))_{t \geq 0}$ be a $(\alpha,\beta)$-resolvent operator function on $\Omega$. Let $\delta \geq 1$, $w \geq 0$ be constants such that 
\begin{equation}\label{eq14}
||\mathbb{S}_{\alpha,\beta}(t)|| \leq \delta\, e^{wt}\, \text{and}\,\, ||\mathbb{T}_{\alpha}(t)||\leq \delta\, e^{wt}
\end{equation}
for all $t \geq 0$. Let $\xi_0 \in\Omega$, $T > 0$ and $G:[0,T] \to (0,\infty)$ be a continuous function.

Suppose that a continuous function $f:[0,T] \to \Omega$ satisfies
\begin{equation}\label{4.1}
\left|\left| t^{1-\gamma}(f(t) - \mathbb{S}_{\alpha,\beta}(t) \xi_0 - \int_0^t \mathbb{T}_{\alpha}(t-s) u(s) \mathcal{H}(s,f(s)) \, {\rm{d}}s )\right|\right| \leq G(t)
\end{equation}
for all $t \in [0,T]$.

Suppose that there exists a positive constant ${\sf \rho}$ such that
\begin{equation}\label{4.2}
\ell(s) |u(s)| e^{w(T-s)} \leq {\sf \rho}
\end{equation}
for almost all $s \in [0,T]$. Then, $\exists ~C_{G} > 0$ (constant) and a unique continuous function $v:[0,T] \rightarrow \rightarrow \Omega$ such that
\begin{equation}\label{4.3}
v(t) = \mathbb{S}_{\alpha,\beta}(t) \xi_0 + \int_0^t \mathbb{T}_{\alpha}(t-s) u(s) \mathcal{H}(s,v(s))\, {\rm{d}}s \, , \quad t \in [0,T]
\end{equation}
and 
\begin{equation}\label{4.4}
||f(t) - v(t)|| \leq C_{G} G(t) \, , \quad t \in [0,T] \, \cdot
\end{equation}
\end{theorem}

\begin{proof}  Consider the constant $K > 0$  such that $\delta \, \rho\, K\, T^{1-\gamma}<1$ and chose a continuous function $\phi:[0,T] \to (0,\infty)$ as follows,
\begin{equation}\label{4.6}
\int_0^t \phi(s) \, {\rm{d}}s \leq K\, \phi(t)\, , \quad t \in [0,T].
\end{equation}

Now let, $f$, $G$ satisfy the inequality {\rm{(\ref{4.1})}} and let ~$\widetilde{\alpha}_G,~\widetilde{\beta}_G>0$ such that
\begin{equation}\label{4.7}
\widetilde{\alpha}_G \phi(t) \leq G(t) \leq \widetilde{\beta}_G \phi(t) \, , \quad t \in [0,T] \, \cdot
\end{equation}

On the other hand, for all $h,g \in C_{1-\gamma}(I,\Omega)$, consider the following set
\begin{equation*}
{\rm{d}}_{\phi,1-\gamma}(h,g) : = {\rm{inf}}\left\{C \in [0,\infty): ||t^{1-\gamma}\left(h(t) - g(t)\right) || \leq C \phi(t) \, , \quad t \in [0,T] \right\}\, \cdot
\end{equation*}

It is easy to see that $(C_{1-\gamma}(I,\Omega),{\rm{d}}_{\phi,1-\gamma})$ is a metric and that $(C_{1-\gamma}(I,\Omega),{\rm{d}}_{\phi,1-\gamma})$ is a complete metric space.

Now, consider the operator $\Lambda : C_{1-\gamma}(I,\Omega) \to C_{1-\gamma}(I,\Omega)$ defined by
\begin{equation*}
(\Lambda h)(t):= \mathbb{S}_{\alpha,\beta}(t) \xi_0 + \int_0^t \mathbb{T}_{\alpha}(t-s) u(s) \mathcal{H}(s,h(s))\, {\rm{d}}s\, , \quad t \in [0,T] \cdot
\end{equation*}

The next step is to show that $\Lambda $ is a contraction in the metric space $C_{1-\gamma}(I,\Omega)$ induced by metric ${\rm{d}}_{\phi,1-\gamma}$. Then, let $h,g \in C_{1-\gamma}(I,\Omega)$ and $C(h,g) \in [0,\infty)$ a constant such that
\begin{equation*}
||t^{1-\gamma}\left(h(t) - g(t)\right)|| \leq C(h,g) \phi(t) \, , \quad t \in [0,T]\, \cdot
\end{equation*}
Then, using {\rm{Eq.(\ref{eq14})}}, {\rm{Eq.(\ref{4.2})}} and {\rm{Eq.(\ref{4.6})}}, we obtain
\begin{equation*}
\begin{array}{rll}
||(\Lambda h)(t) - (\Lambda g)(t)|| & = & \displaystyle \left|\left| \int_0^t \mathbb{T}_{\alpha}(t-s) u(s) \left( \mathcal{H}(s,h(s)) - \mathcal{H}(s,g(s)) \right)\, 
{\rm{d}}s \right|\right|\\
& \leq & \displaystyle  \int_0^t ||\mathbb{T}_{\alpha}(t-s)|| |u(s)| ||\mathcal{H}(s,h(s)) - \mathcal{H}(s,g(s))|| \, {\rm{d}}s \\
&\leq & \displaystyle  \int_0^t \delta\, e^{w(t-s)} |u(s)| \ell(s) ||h-g||_{C_{1-\gamma}} \, {\rm{d}}s \\
& \leq & \displaystyle  \delta C(h,g) \int_0^t  e^{w(t-s)} \phi(s) |u(s)| \ell(s) \, {\rm{d}}s \\
& \leq & \displaystyle  \delta C(h,g) \rho \int_0^t  \phi(s) \, {\rm{d}}s \\
& \leq & T^{1-\gamma}\delta C(h,g)\rho K \phi(t) \, , \quad t \in [0,T] \cdot
\end{array}
\end{equation*}

Therefore, we have $d_{\phi,1-\gamma}(\Lambda(h),\Lambda(g)) \leq \delta \, \rho \, K \,T ^ {1-\gamma}\, C(h,g)$ from which we deduce that
\begin{equation*}
d_{\phi,1-\gamma}(\Lambda(h),\Lambda(g)) \leq \delta \, \rho \, K \,T ^ {1-\gamma}\,  d_{\phi,1-\gamma}(h,g)  \cdot
\end{equation*}

Using the fact that $\delta \, \rho \, K \,T ^ {1-\gamma}< 1$, we have that $\Lambda$ is a contraction in $(C_{1-\gamma}(I,\Omega),d_{\phi,1-\gamma})$. In this sense, through Banach's fixed point theorem, we have that there is a unique function $v \in C_{1-\gamma}(I,\Omega)$ such that $v=\Lambda(v)$. Now, using By the triangle inequality, we get
\begin{equation*}
\begin{array}{rll}
d_{\phi,1-\gamma}(f,v) & \leq & d_{\phi,1-\gamma}(f,\Lambda(f)) + d_{\phi,1-\gamma}(\Lambda(f),\Lambda(v))\\
& \leq & \beta_{G} + \delta \, \rho \, K \,T ^ {1-\gamma}\, d_{\phi,1-\gamma}(f,v)
\end{array}
\end{equation*}
which implies that
\begin{equation}
d_{\phi,1-\gamma}(f,v) \leq \frac{\beta_G}{1-\delta \, \rho \, K \,T ^ {1-\gamma}}.
\end{equation}

Which in turn, gives
\begin{equation}\label{eq21}
\begin{array}{rll}
||t^{1-\gamma}\left(f(t) - v(t)\right) ||& \leq & \displaystyle \frac{\beta_G}{1-\delta \, \rho \, K \,T ^ {1-\gamma}} \phi(t)\\
&&\\
& \leq & \displaystyle \frac{\beta_G}{1-\delta \, \rho \, K \,T ^ {1-\gamma}} \frac{G(t)}{\alpha_G} = C_G G(t) \, , \quad t \in [0,T]
\end{array}
\end{equation}
where $C_G:= \displaystyle \frac{\beta_G}{(1-\delta \, \rho \, K \,T ^ {1-\gamma}) \, \alpha_G},$ which is the desired inequality {\rm{(\ref{4.4})}}.
\end{proof}

\begin{rem}
From {\rm Theorem 1} and {\rm Theorem 2}, we have some particular cases, that is, by taking the boundaries with $\beta \rightarrow 1 $ and $ \beta \rightarrow 0$. Also, we have the whole case when $ \alpha = 1 $. So we have the following versions:

{\rm (1)} Taking $\beta \rightarrow 0$ in {\rm Eq.(6)}, we have as particular case, the version of Theorem 1 for the Riemann-Liouville fractional derivative, given by:

\begin{theorem} Let $\left( S_{\alpha ,0}\left( t\right) \right) _{t\geq 0}$ the $\left( \alpha ,0\right) -$ resolvent operator function on Banach space $\left( \Omega ,\left\Vert {\cdot}\right\Vert \right) $ with and let $T>0$ be a positive real number. We set
\begin{equation*}
\widetilde{\lambda }:=\delta \, T^{1-\gamma} \int_{0}^{T}e^{\omega \left( T-s\right)
}\left\vert u\left( s\right) \right\vert \ell \left( s\right) ds.
\end{equation*}
If $\widetilde{\lambda }<1$ them {\rm Eq.(6)} is stable in the Ulam-Hyers sense.
\end{theorem} 

{\rm (2)} Taking limit $\beta \rightarrow 1$ in {\rm Eq.(6)}, we have the version of Theorem 1 for the Caputo fractional derivative, ensuring that {\rm Eq.(6)} is Ulam-Hyers stable.

{\rm (3)} Taking limit $\beta \rightarrow 1$ in {\rm Eq.(6)}, we have as particular case the version of {\rm Theorem 2} for the Caputo fractional derivative given by the following theorem (Ulam-Hyers-Rassias):

\begin{theorem} Let $\left( \Omega ,\left\Vert {\cdot}\right\Vert \right) $ be a Banach space and $\left( \mathbb{S}_{\alpha ,1}\left( t\right) \right) _{t\geq 0}$ be $\left( \alpha ,1\right) -$ resolvent operator function on $\Omega $. Let $\delta \geq 1$, $\omega \geq 0$ be constants such that $\left\Vert \mathbb{S}_{\alpha ,1}\left( t\right) \right\Vert \leq \delta e^{\omega t}$ and $\left\Vert \mathbb{T}_{\alpha }\left( t\right) \right\Vert \leq \delta e^{\omega t}$ for all $ t\geq 0$. Let $\xi _{0}\in \Omega $ be fixed, $T>0$ and $G:\left[ 0,T\right] \rightarrow \left( 0,\infty \right) $ be a continuous function. Suppose that a continuous function $f:\left[ 0,T\right] \rightarrow \Omega $ satisfies
\begin{equation*}
\left\Vert t^{1-\gamma}(f\left( t\right) -\mathbb{S}_{\alpha ,1}\left( t\right) \xi _{0}-\int_{0}^{t}\mathbb{T}_{\alpha }\left( t-s\right) u\left( s\right) H\left( s,f\left( s\right) \right) ds)\right\Vert \leq G\left( t\right) 
\end{equation*}
for all $t\in \left[ 0,T\right]$.

Suppose that exists a positive constant $\rho $ such that 
\begin{equation*}
\ell \left( s\right) \left\vert u\left( s\right) \right\vert e^{\omega
\left( T-s\right) }\leq \rho 
\end{equation*}
for almost all $s\in \left[ 0,T\right]$. Then, exist the constant $C_{G}>0$ and a unique continuous functions $v:\left[ 0,T\right] \rightarrow \Omega $ such that
\begin{equation*}
v\left( t\right) =\mathbb{S}_{\alpha ,1}\left( t\right) \xi
_{0}+\int_{0}^{t}\mathbb{T}_{\alpha }\left( t-s\right) u\left( s\right) H\left(
s,f\left( s\right) \right) ds,\text{ }t\in \left[ 0,T\right] 
\end{equation*}
and
\begin{equation*}
\left\Vert t^{1-\gamma}(f\left( t\right) -v\left( t\right)) \right\Vert \leq C_{G}G\left(t\right) ,\text{ }\forall t\in \left[ 0,T\right] .
\end{equation*}
\end{theorem}

4. Taking limit $\beta \rightarrow 1$ or $\beta \rightarrow 0$ and choosing $\alpha =1$, we have the version of the {\rm Theorem 1} and {\rm Theorem 2}, for integer case.
\end{rem}

\section{Ulam-Hyers and Ulam-Hyers-Rassias stabilities of mild solution on $[0,+\infty)$}

As in section 3, we will investigate the Ulam-Hyers and Ulam-Hyers-Rassias stabilities in the interval $ [0, +\infty]$, 
with the same assumption on the function $\mathcal{H}$. So we start with the following theorem:

\begin{theorem}Let $\xi_0 \in \Omega$ be fixed and let $\varepsilon > 0$ be a given positive number. Suppose that a continuous function $f:[0,+\infty) \to \Omega$ satisfies 
\begin{equation}\label{5.1}
\left|\left|t^{1-\gamma}( f(t) - \mathbb{S}_{\alpha,\beta}(t) \xi_0 - \int_0^t \mathbb{T}_{\alpha}(t-s) u(s) \mathcal{H}(s,f(s)) \, {\rm{d}}s) \right|\right| \leq \varepsilon
\end{equation}
for all $t \in [0,+\infty)$.

Suppose that 
\begin{equation}\label{5.2}
\widetilde{\lambda}_{\alpha,1-\gamma} = \displaystyle \sup_{t \geq 0}\, t^{1-\gamma} \,\int_0^t \ell(s) |u(s)| ||\mathbb{T}_{\alpha}(t-s)|| \, {\rm{d}}s < 1
\end{equation}
with $0 < \alpha \leq 1$ and $0 \leq \gamma \leq 1$.

Then, there exists a unique continuous function $v:[0,+\infty) \to \Omega$ such that
\begin{equation}\label{5.3}
v(t) = \mathbb{S}_{\alpha,\beta}(t) \xi_0 + \int_0^t u(s) \mathbb{T}_{\alpha}(t-s) \mathcal{H}(s,v(s))\, {\rm{d}}s \, , \quad t \in [0,+\infty)
\end{equation}
and 
\begin{equation}\label{5.4}
||t^{1-\gamma}\left(f(t) - v(t)\right) || \leq \frac{\varepsilon}{1-\widetilde{\lambda}_{\alpha,1-\gamma}} \, , \quad t \in [0,+\infty) \, \cdot
\end{equation}
\end{theorem}

\begin{proof}

Consider that $\widetilde{\lambda}_{\alpha,1-\gamma} < 1$, $\varepsilon > 0$ be given and $f \in C_{1-\gamma}([0,+\infty),\Omega)$ satisfy the inequality {\rm{(\ref{5.1})}}. On the other hand, we consider the set $\widetilde{\cal E}_{f,1-\gamma}$, given by
\begin{equation}
\widetilde{\cal E}_{f,1-\gamma}:= \left\{ g \in C_{1-\gamma}([0,+\infty),\Omega); \,\, \displaystyle \sup_{t \geq 0} ||t^{1-\gamma}\left(g(t) - f(t)\right) || < + \infty \right\}.
\end{equation}

The set $\widetilde{\cal E}_{f,1-\gamma}$ is not empty, because it contains $f$ and $\Lambda(f)$. Now, consider the functions $h,g \in \widetilde{\cal E}_{f,1-\gamma}$, such that
\begin{equation*}
d_{1-\gamma}(h,g):= \displaystyle \sup_{t \geq 0} ||t^{1-\gamma}\left(h(t) - g(t)\right) ||.
\end{equation*}

Then, $d_{1-\gamma}$ is a distance and the metric space $(\widetilde{\cal E}_{f,1-\gamma},d_{1-\gamma})$ is complete. 

For any functions $h,g \in \widetilde{\cal E}_{f,1-\gamma}$, we get
\begin{equation*}
\begin{array}{rll}
||(\Lambda h)(t) - (\Lambda g)(t)|| & = & \displaystyle \left|\left| \int_0^t u(s) \mathbb{T}_{\alpha}(t-s) \left[ \mathcal{H}(s,h(s)) - \mathcal{H}(s,g(s)) \right]\, 
{\rm{d}}s \right|\right|\\
& \leq & \displaystyle  \int_0^t ||\mathbb{T}_{\alpha}(t-s)|||u(s)| ||\mathcal{H}(s,h(s)) - \mathcal{H}(s,g(s))|| \, {\rm{d}}s \\
& \leq & \left( \displaystyle  \int_0^t ||\mathbb{T}_{\alpha}(t-s)|| |u(s)| \ell(s) \, {\rm{d}}s \right) \, d_{1-\gamma}(h,g)\, , \quad t \in [0,+\infty).
\end{array}
\end{equation*}
This gives 
\begin{equation*}
||t^{1-\gamma} \left( (\Lambda h)(t) - (\Lambda g)(t)\right) || \leq  \widetilde{\lambda}_{\alpha,1-\gamma} d_{1-\gamma}(h,g)\, , \quad t \in [0,+\infty).
\end{equation*}
Therefore, we have
\begin{equation*}
d_{1-\gamma}(\Lambda h, \Lambda g) \leq \widetilde{\lambda}_{\alpha,1-\gamma} d_{1-\gamma}(h,g).
\end{equation*}

Moreover, it is easy to show that $\Lambda(h) \in \widetilde{\cal E}_{f,1-\gamma}$ for any function $h \in \widetilde{\cal E}_{f,1-\gamma}$. Thus, we have $\Lambda $ is a contraction in $(\widetilde{\cal E}_{f,1-\gamma},d_{1-\gamma})$. In this sense, by Banach's fixed point theorem, we have that there is only one element $v \in \widetilde{\cal E}_{f,1-\gamma}$ such that $v=\Lambda(v)$. By the triangle inequality, we get
\begin{equation*}
\begin{array}{rll}
d_{1-\gamma}(f,v) & \leq & \leq d_{1-\gamma}(f,\Lambda(f)) +  d_{1-\gamma}(\Lambda(f),\Lambda(v))\\ 
& \leq & \varepsilon + \widetilde{\lambda}_{\alpha,1-\gamma} d_{1-\gamma}(f,v) 
\end{array}
\end{equation*}
that implies
\begin{equation*}
d_{1-\gamma}(f,v) \leq \frac{\varepsilon}{1-\widetilde{\lambda}_{\alpha,1-\gamma}},
\end{equation*}
this is, \begin{equation}\label{5.5}
||t^{1-\gamma}\left( f(t) - v(t)\right) ||\leq c\, \varepsilon\, , \quad t \in [0,+\infty)
\end{equation}
where $c:=\dfrac{1}{1-\widetilde{\lambda}_{\alpha,1-\gamma}}$ .

The inequality {\rm{(\ref{5.5})}} shows that the {\rm{Eq.(\ref{2.1})}} is Ulam-Hyers stable.
\end{proof}

With the following result aimed at investigating the Ulam-Hyers-Rassias stability, complete the second main result of this paper.

\begin{theorem}
Let $\Omega$ be a Banach space, $(\mathbb{S}_{\alpha,\beta}(t))_{t \geq 0}$ be a $(\alpha,\beta)$-resolvent operator function on $\Omega$ and $\phi_0 \in \Omega$ be fixed. Let $K >0$ be given and $\phi:[0,+\infty) \to (0,+\infty)$ be a continuous function such that 
\begin{equation}\label{6.1}
\int_0^t \phi(s) \, {\rm{d}}s \leq K \, \phi(t)\, , \quad t \in [0,+\infty) \cdot
\end{equation}

Suppose that a continuous function $f:[0,+\infty) \to \Omega$ satisfies
\begin{equation}\label{6.2}
\left|\left| t^{1-\gamma}(f(t) - \mathbb{S}_{\alpha,\beta}(t) \xi_0 - \int_0^t u(s) \mathbb{T}_{\alpha}(t-s)  \mathcal{H}(s,f(s)) \, {\rm{d}}s) \right|\right| \leq \phi(t)
\end{equation}
for all $t \in [0,+\infty)$.

Suppose that there exists a positive constant $\rho > 0$ such that
\begin{equation}\label{6.3}
\ell(s) |u(s)| ||\mathbb{T}_{\alpha}(t-s)||\leq \rho
\end{equation}
for almost all $(s,t) \in [0,+\infty)$ with $0 \leq s \leq t$ and suppose that 
\begin{equation}\label{6.4}
\rho \,K \, T^{1-\gamma}  < 1 \cdot
\end{equation}

Then, there exists a unique continuous function $v:[0,+\infty) \to \omega$ such that
\begin{equation}\label{6.5}
v(t) = \mathbb{S}_{\alpha,\beta}(t) \xi_0 + \int_0^t u(s) \mathbb{T}_{\alpha}(t-s) \mathcal{H}(s,v(s))\, {\rm{d}}s \, , \quad t \in [0,+\infty)
\end{equation}
and 
\begin{equation}\label{6.6}
||t^{1-\gamma}\left( f(t) - v(t)\right) || \leq \frac{1}{1-K\rho} \phi (t) \, , \quad t \in [0,+\infty) \cdot
\end{equation}
\end{theorem}

\begin{proof} Let $f \in C_{1-\gamma}([0,+\infty),\Omega)$ satisfy the inequality {\rm{(\ref{6.2})}} and the following set, defined by
\begin{equation*}
\widetilde{\cal E}_{f,1-\gamma}:= \left\{ g \in C_{1-\gamma}([0,+\infty),\Omega): \exists C \geq 0 : ||t^{1-\gamma}\left( g(t)-f(t)\right) || \leq C \phi(t) \, , \quad t \in [0,+\infty) \right\}.
\end{equation*}
The set $\widetilde{\cal E}_{f,1-\gamma}$ is not empty, because it contains $f$ and $\Lambda(f)$.  Now, for $h,g \in \widetilde{\cal E}_{f,1-\gamma}$, we define the following set
\begin{equation*}
d_{\phi,1-\gamma}(h,g):= \inf \left\{ C \in [0,+\infty): ||t^{1-\gamma}\left( h(t)-g(t)\right) || \leq C \phi(t)\, , \quad t \in [0,+\infty)  \right\}.
\end{equation*}

Note that it is easy to see that $(\widetilde{\cal E}_{f,1-\gamma},d_{\phi,1-\gamma})$ is a complete metric space satisfying $ \Lambda(\widetilde{\cal E}_{f,1-\gamma}) \subset  \widetilde{\cal E}_{f,1-\gamma}$, where $\Lambda : \widetilde{\cal E}_{f,1-\gamma} \rightarrow \widetilde{\cal E}_{f,1-\gamma}$ 
is defined by 
\begin{equation*}
(\Lambda h)(t):= \mathbb{S}_{\alpha,\beta}(t) \xi_0 + \int_0^t u(s) \mathbb{T}_{\alpha}(t-s) \mathcal{H}(s,h(s))\, {\rm{d}}s \, , \quad t \in [0,+\infty).
\end{equation*}

The idea is to prove that in fact the $\Lambda$ application is a contraction on the metric space $(\widetilde{\cal E}_{f,1-\gamma},{\rm{d}}_{\phi,1-\gamma})$. Then, let $h,g \in \widetilde{\cal E}_{f,1-\gamma}$ and $C(h,g) \in [0,+\infty)$ be an arbitrary constant such that
\begin{equation*}
||t^{1-\gamma}\left( h(t) - g(t)\right) || \leq C(h,g) \phi(t) \, , \quad t \in [0,+\infty).
\end{equation*}

In this sense, we have the following inequality
\begin{equation*}
\begin{array}{rll}
||(\Lambda h)(t) - (\Lambda g)(t)||& = & \displaystyle \left|\left| \int_0^t u(s) \mathbb{T}_{\alpha}(t-s) \left( \mathcal{H}(s,h(s)) - \mathcal{H}(s,g(s)) \right)\, 
{\rm{d}}s \right|\right|\\
& \leq & \displaystyle  \int_0^t |u(s)| ||\mathbb{T}_{\alpha}(t-s)||_{C_{1-\gamma}} ||\mathcal{H}(s,h(s)) - \mathcal{H}(s,g(s))|| \, {\rm{d}}s \\
&\leq & \displaystyle  \int_0^t |u(s)| ||\mathbb{T}_{\alpha}(t-s)||_{C_{1-\gamma}} \ell(s) ||h - g||_{C_{1-\gamma}} \, {\rm{d}}s \\
& \leq & \displaystyle  C(h,g) \int_0^t  |u(s)| \ell(s) ||\mathbb{T}_{\alpha}(t-s)|| \phi(s) \, {\rm{d}}s \\
& \leq & \displaystyle  \rho \,  C(h,g)  \int_0^t  \phi(s) \, {\rm{d}}s \\
& \leq & \rho \, C(h,g)\, K \phi(t) \, , \quad t \in [0,T] \cdot
\end{array}
\end{equation*}

Therefore, we have ${\rm{d}}_{\phi,1-\gamma}(\Lambda(h),\Lambda(g)) \leq C(h,g) \, T^{1-\gamma} \,\rho\, K $, that implies in
\begin{equation*}
{\rm{d}}_{\phi,1-\gamma}(\Lambda(h),\Lambda(g)) \leq T^{1-\gamma} \,\rho\, K {\rm{d}}_{\phi,1-\gamma}(h,g).
\end{equation*}

Using the fact that $T^{1-\gamma} \,\rho\, K < 1$, we get $\Lambda$ is strictly contractive on the $(\widetilde{\cal E}_{f,1-\gamma},{\rm{d}}_{\phi,1-\gamma})$. Thus, through the Banach fixed-point theorem, there is a unique function $v\in\widetilde{\cal E}_{f,1-\gamma}$ such that $v=\Lambda(v)$. Using the triangle inequality, we obtain
\begin{equation*}
\begin{array}{rll}
d_{\phi,1-\gamma}(f,v) & \leq & d_{\phi,1-\gamma}(f,\Lambda(f)) + d_{\phi,1-\gamma}(\Lambda(f),\Lambda(v))\\
& \leq & 1 + T^{1-\gamma} \,\rho\, K\, d_{\phi,1-\gamma}(f,v)
\end{array}
\end{equation*}
which implies that
\begin{equation*}
d_{\phi,1-\gamma} \leq \frac{1}{1-T^{1-\gamma} \,\rho\, K}.
\end{equation*}

Therefore, we conclude that
\begin{equation*}
||t^{1-\gamma} \left( f(t) - v(t)\right) ||\leq C_{\phi} \phi(t)\, , \quad t \in [0,+\infty)
\end{equation*}
where $C_{\phi}:= \dfrac{1}{1-T^{1-\gamma} \,\rho\, K}$.
\end{proof}

Remark 2. In the same way that we highlight the particular cases for Theorem 1 and Theorem 2, here also the observation made according to remark 1 is valid.


\section{Concluding remarks}
We conclude this paper with the objectives achieved, that is, we investigate the  Ulam-Hyers and Ulam-Hyers-Rassias stabilities for the mild solution of the fractional nonlinear abstract non-linear Cauchy problem: the first part was destined to the inite interval $ [0, T] $ and the second part to the red infinite interval $[0, \infty) $. It is important to emphasize the fundamental role of the Banach fixed point theorem in the results obtained.

Although, the results presented here, contribute to the growth of the theory; some questions still need to be answered. The first question is about the possibility of investigating the existence and uniqueness of mild solutions for fractional differential equations formulated via $\psi$-Hilfer fractional derivative. Consequently, the second allows us to question the Ulam-Hyers stabilities. But for such a success, it is necessary and sufficient condition to obtain a Laplace transform and inverse Laplace transform with respect to another function \cite{jarad2}.

Another important consequence that we can take, the mild solution part of a fractional problem, is to be able to investigate properties of Navier-Stokes equations \cite{neto,soj}. So this path is the next step in the research being developed.

\section*{Acknowledgment}

JVCS acknowledges the financial support of a PNPD-CAPES (process number nº88882.305834/2018-01) scholarship of the Postgraduate Program in Applied Mathematics of IMECC-Unicamp. The second author acknowledges the Science and Engineering Research Board (SERB), New Delhi, India for the Research Grant (Ref: File no. EEQ/2018/000407).


\end{document}